\documentclass{amsart}
\usepackage{mathrsfs}
\usepackage{amssymb,latexsym,amsmath,amsthm,enumitem,mathrsfs}
\usepackage{hyperref}
\usepackage{color}
\usepackage{stmaryrd}
\usepackage{mathrsfs}
\usepackage{lineno}
\usepackage{bbm}
\usepackage{scalerel,stackengine}
\usepackage{tikz}
\usepackage{enumitem}
\usepackage{ esint }
\usepackage{graphicx}
\usepackage{amssymb}
\newtheorem{theorem}{Theorem}
\newtheorem{conjecture}[theorem]{Conjecture}
\newtheorem{lemma}[theorem]{Lemma}

\newtheorem{remark}[theorem]{Remark}
\newtheorem{proposition}[theorem]{Proposition}

\newtheorem{definition}[theorem]{Definition}

\theoremstyle{definition}

\newcommand{\cC}{\mathcal{C}}
\newcommand{\cD}{\mathcal{D}}

\newcommand{\cH}{\mathcal{H}}
\newcommand{\cI}{\mathcal{I}}

\newcommand{\cM}{\mathcal{M}}
\newcommand{\cN}{\mathcal{N}}

\newcommand{\cV}{\mathcal{V}}

\newcommand{\bA}{\mathbf{A}}

\newcommand{\bH}{\mathbf{H}}

\newcommand{\bR}{\mathbf{R}}
\newcommand{\bS}{\mathbf{S}}
\newcommand{\bT}{\mathbf{T}}

\newcommand{\bW}{\mathbf{W}}

\newcommand{\T}{\mathbb T}

\newcommand{\e}{\varepsilon}

\def\set4{\mathcal I}
\def\tup14{(1,2,3,4)}

\newtheorem*{comm*}{Comment}
\newtheorem*{lemma*}{Lemma}

\newcommand{\supp}{\mathrm{supp}}

\newcommand{\R}{\mathbb{R}}
\newcommand{\D}{\mathbb D}
\newcommand{\de}{\delta} 

\newcommand{\wh}{\widehat}
\newcommand{\wt}{\widetilde}

\newcommand{\en}{\epsilon_{\circ}}

\newcommand{\al}{\alpha}
\newcommand{\dir}{\textup{dir}}
\newcommand{\tx}{\textup}
\newcommand{\cQ}{\mathcal Q}

\DeclareMathOperator*{\esssup}{ess\,sup}
\newcommand{\Bush}{\textup{Bush}}

\newcommand{\cP}{\mathcal{P}}
\newcommand{\cF}{\mathcal{F}}

\newcommand{\G}{\mathbb{G}}
\newcommand{\Aloc}{A_{\text{loc}}}
\newcommand{\bell}{\boldsymbol{\ell}}

\begin{document}

 \author{Shengwen Gan}
 \address{Department of Mathematics\\
 University of Wisconsin-Madison\\
 Madison, WI-53706, USA}
 \email{sgan7@wisc.edu}

\keywords{Marstrand projection theorem, exceptional set estimate, Fourier analysis, Brascamp-Lieb inequality}
\subjclass[2020]{28A75, 28A78}

\date{}

\title{A Marstrand projection theorem for lines}
\maketitle

\begin{abstract}
Fix integers $1<k<n$.
For $V\in G(k,n)$, let $P_V: \mathbb{R}^n\rightarrow V$ be the orthogonal projection. For $V\in G(k,n)$, define the map
\[ \pi_V: A(1,n)\rightarrow  A(1,V)\bigsqcup V. \]
\[ \ell\mapsto P_V(\ell). \]

For any $0<a<\textup{dim}(A(1,n))$, we find the optimal number $s(a)$ such that the following is true. For any Borel set $\bA \subset A(1,n)$ with $\dim(\bA)=a$, we have
\[ \dim(\pi_V(\bA))=s(a), \textup{~for~a.e.~}V\in G(k,n). \]
When $A(1,n)$ is replaced by $A(0,n)=\R^n$, it is the classical Marstrand  projection theorem, for which $s(a)=\min\{k,a\}$. A new ingredient of the paper is the Fourier transform on affine Grassmannian.
\end{abstract}

\section{Introduction}
For $k<n$,
let $G(k,n)$ be the set of $k$-dimensional subspaces in $\R^n$, and $A(k,n)$ be the set of $k$-dimensional affine spaces in $\R^n$. For $V\in G(k,n)$ and $l<k$, let $G(l,V)$ be the set of $l$-dimensional subspaces in $V$, and $A(l,V)$ be the set of $l$-dimensional affine spaces in $V$. For simplicity, we just call $l$-dimensional affine space as $l$-plane. Fix integers $1<k<n$. For $V\in G(k,n)$, let 
\[ P_V:\R^n\rightarrow V \]
be the orthogonal projection onto $V$. 

Let $V\in G(k,n)$.
Note that for $L\in A(1,n)$, we have $P_V(L)$ is either a line or a point in $V$. We can define
\[ \pi_V: A(1,n)\rightarrow  A(1,V)\bigsqcup V \]
\[ L\mapsto P_V(L). \]

Marstrand \cite{marstrand1954some} proved the following result. Let $A\subset \R^n$ with $\dim (A)=a$. Then
\begin{equation}\label{classical}
    \dim(P_V(A))=\min\{a,k\}, \textup{~for~a.e.~} V\in G(k,n). 
\end{equation} 
Naturally, one can consider the following Marstrand-type projection problem for $A(1,n)$. 

For $0<a<2(n-1)=\dim(A(1,n))$, what is the optimal number $s(a)$ such that the following is true?
Let $\bA\subset A(1,n)$ with $\dim(\bA)=a$. (Since $A(1,n)$ is a Riemannian manifold, $\dim(\bA)$ is naturally defined.) Then we have
\begin{equation}
    \dim(\pi_V(\bA))=s(a), \textup{~for~a.e.~} V\in G(k,n).
\end{equation}

It is not hard to see that 
\[ s(a)\le \min\{a, \dim(A(1,k))\}=\min\{a,2(k-1)\}.\] 
At the first glance, one may guess the optimal number $s(a)$ is $\min\{a,2(k-1)\}$ which has the same form as in \eqref{classical}. However, this is not true as shown in the following example. 

Consider $n=3,k=2$. Let $\bA=G(1,3)$ which is the set of lines passing through the origin.
We see that for any $V\in G(2,3)$, $\pi_V(\bA)=G(1,V)\bigsqcup \{0\}$ which is one-dimensional set of lines plus a point. We get 
\[ \dim(\pi_V(\bA))=1<2=\min\{ \dim(\bA), 2(k-1)\}. \]

We see that the problem becomes quite different when we consider the projection of lines, compared with the projection of points.
One reason is that distinct lines can have overlaps while distinct points do not. Because of this overlapping phenomenon, we are able to stack lines in $\R^n$ so that their projections to subspaces have some different structure. This makes the line version of Marstrand projection problem harder than the point version.

\begin{definition}
    Fix $1< k<n$. For any $0<a<\dim(A(1,n))$, define
    \begin{equation}\label{defSa}
        S(a):=\inf_{\bA\subset A(1,n), \dim(\bA)=a}\esssup_{V\in G(k,n)}\dim(\pi_V(\bA)).
    \end{equation}
Here $\esssup$ is with respect to the unique probability measure on $G(k,n)$ which is invariant under rotation. 
In our paper, we always assume $\bA$ to be a Borel set to avoid some measurability issue. We also remark that $S(a)=S_{k,n}(a)$ should also depend on $k,n$, but we just drop them from the notation for simplicity.
\end{definition}

We state our main theorem.

\begin{theorem}\label{mainthm}
    We have the following exact value of $S(a)$.
    \begin{align}
    &S(a)= a,\ \ \ a\in [0,k-1],\\
    &S(a)=k-1,\ \ \ a\in [k-1,n-1],\\
    &S(a)=a-(n-k),\ \ \ a\in [n-1,n+k-2],\\
    &S(a)=2(k-1),\ \ \ a\in [n+k-2, 2(n-1)].
\end{align}
\end{theorem}

\begin{remark}
    {\rm
    We can also write $S(a)$ as
     \begin{align}
    &S(a)= \min\{a,k-1\},\ \ \ a\in [0,n-1],\\
    &S(a)=\min\{a-(n-k),2(k-1)\},\ \ \ a\in [n-1, 2(n-1)].
    \end{align}
    }
\end{remark}

We will prove Theorem \ref{mainthm} by showing the following two propositions.

\begin{proposition}\label{upperprop}
    We have the following upper bounds of $S(a)$. 
\begin{equation}\label{upperineq}
    \begin{split}
    &S(a)\le \min\{a,k-1\},\ \ \ a\in [0,n-1],\\
    &S(a)\le \min\{a-(n-k),2(k-1)\},\ \ \ a\in [n-1, 2(n-1)].
    \end{split}
\end{equation}
\end{proposition}

\begin{proposition}\label{lowerprop}
    We have the following lower bounds of $S(a)$. 
\begin{align}
\label{lowerineq1}&S(a)\ge \min\{a,k-1\},\ \ \ a\in [0,n-1],\\
    \label{lowerineq2}&S(a)\ge \min\{a-(n-k),2(k-1)\},\ \ \ a\in [n-1, 2(n-1)].
\end{align}
\end{proposition}

The proof of Proposition \ref{upperprop} is short and provides good examples. We give it here.
\begin{proof}[Proof of Proposition \ref{upperprop}]
    When $a\in [0,n-1]$, we choose $\bA$ to be a subset of $G(1,n)$ with dimension $a$. We see that for any $V \in G(k,n)$, $\pi_V(\bA)\subset G(1,V)\bigsqcup \{0\}$. Therefore, 
    \[ \dim(\pi_V(\bA))\le \min\{a,k-1\} \]
    for any $V\in G(k,n)$, and hence
    \[ S(a)\le \min\{a,k-1\}. \]

    When $a\in [n-1,2(n-1)]$, we write $a=n-1+\beta$. Choose $A\subset \R^{n-1}$ (here $\R^{n-1}$ is spanned by the first $n-1$ coordinates) to be a $\beta$-dimensional set. For each $x\in A$, let $\bA_x$ be the set of lines passing through $x$ and transverse to $\R^{n-1}$. More precisely,
    \[\bA_x =\{\ell\in A(1,n): x\in \ell, \ell\not\subset \R^{n-1}   \}. \]
    Choose $\bA=\bigsqcup_{x\in A} \bA_x$. Since $\{\bA_x\}_{x\in A}$ are disjoint and $\bA$ has a product structure, we have
    \[ \dim(\bA)=\dim A+\dim(\bA_x)=\beta+n-1=a. \]
    Note that if $V\in G(k,n)$ and $V$ does not contain the $x_n$-axis, then $\pi_V(\bA_x)=\big(P_V(x)+G(1,V)\big)\bigsqcup \{P_V(x)\}$. If $V\in G(k,n)$ contains the $x_n$-axis, then $\pi_V(\bA_x)=P_V(x)+G(1,V)$.
    As a result, $\pi_V(\bA)=\bigcup_{x\in A} \pi_V(\bA_x)$ consists of at most $\beta$-dimensional translated copies of $G(1,V)$ plus at most $\beta$-dimensional set of points. We have
    \[ \dim(\pi_V(\bA))\le \beta+k-1=a-(n-k). \]
    Of course, we also have $\dim(\pi_V(\bA))\le \dim(A(1,V)\bigsqcup V)= 2(k-1)$.
\end{proof}

For Proposition \ref{lowerprop}, we will be able to prove a stronger result known as the exceptional set estimate.

Recall that $P_V:\R^n\rightarrow V$ is the orthogonal projection for a given $V\in G(k,n)$. For a set $A\subset \R^n$ with $\dim A=a$ and a parameter $s$ satisfying $0<s<\min\{a,k\}$, we consider the set
\begin{equation}
    \{ V\in G(k,n): \dim(P_V(A))<s \},
\end{equation}
which is known as the exceptional set. There are two types of the estimates for the exceptional set:

\begin{equation}\label{kauf}
   \dim \bigg( \{ V\in G(k,n): \dim(P_V(A))<s \} \bigg)\le k(n-k)+s-k. 
\end{equation}

\begin{equation}\label{falc}
   \dim \bigg( \{ V\in G(k,n): \dim(P_V(A))<s \} \bigg)\le k(n-k)+s-a; 
\end{equation}

We call the first one Kaufman-type estimate, and the second one Falconer-type estimate. See the reference \cite{kaufman1968hausdorff}, \cite{falconer1982hausdorff}, \cite{peres2000smoothness} and \cite{mattila2015fourier}. It is not hard to see that by letting $s\rightarrow \min\{a,k\}$, either \eqref{falc} or \eqref{kauf} will imply \eqref{classical}.

By the same idea, we can deduce Proposition \ref{lowerprop} from the following two exceptional set estimates. 
\begin{theorem}\label{exthm}
    Fix a number $0<\mu<1/100$. Let $A_\mu$ be a ball of radius $\mu$ in $\Aloc(1,n)$ and $G_\mu$ be a ball of radius $\mu$ in $G(k,n)$, so that for any $\ell\in A_\mu, V\in G_\mu$, we have
    \begin{equation}\label{transcond}
        \angle (\ell,V^\perp)>\mu. 
    \end{equation}
    For $\bA\subset A_\mu$ with $\dim(\bA)=a$, and $0<s<\min\{a,2(k-1)\}$, define the exceptional set
    \begin{equation}
        E_s(\bA):=\{ V\in G_\mu: \dim(\pi_V(\bA))<s \}.
    \end{equation}
We have the following estimates:

\begin{equation}\label{kaufman}
\dim(E_s(\bA))\le k(n-k)+s-(k-1).    
\end{equation}

\begin{equation}\label{falconer}
\dim(E_s(\bA))\le \max\{0,k(n-k)+s-a+(n-k)\}.    
\end{equation}
\end{theorem}

\begin{remark}
    {\rm
    For some technical reason, here we use $A_\mu, G_\mu$ (instead of $A(1,n), G(k,n)$) to make that for any $\ell\in A_\mu$ and $ V\in G_\mu$, $\pi_V(\ell)$ is a line.

    The proof of \eqref{kaufman} relies on a tube-slab incidence estimate. The proof of \eqref{falconer} relies on the Fourier analysis in $A(1,n)$, which is the main novelty of this paper. 

    There is another notable thing. To prove \eqref{classical}, we just need one of \eqref{kauf} or \eqref{falc}. However, to prove Proposition \ref{lowerprop}, we need both \eqref{kaufman} and \eqref{falconer}. Actually, \eqref{kaufman} implies \eqref{lowerineq1}, while \eqref{falconer} implies \eqref{lowerineq2}.

    Peres and Schlag \cite{peres2000smoothness} explored projection problems within a broader context. For an introduction to this method, we recommend \cite[Chapter 18]{mattila2015fourier}. Peres and Schlag introduced the concept of the ``transversality condition" for a family of general projection maps. They also established Marstrand-type estimates and exceptional set estimates when this condition is met.
One might ponder whether we can employ the approach of Peres and Schlag, involving the definition of general projection maps and the verification of the transversality condition, for our specific problem. However, in our case, it does not yield the precise estimate we require.
    }
\end{remark}

In Section \ref{sec2}, we introduce the notation. In Section \ref{sec3}, we show that Theorem \ref{exthm} implies Proposition \ref{lowerprop}. In Section \ref{sec4}, we prove the Kaufman-type estimate \eqref{kaufman}. In Section \ref{sec5}, we introduce the Fourier transform in $A(1,n)$ and then prove the Falconer-type estimate \eqref{falconer}.

\section{Preliminary}\label{sec2}

We discuss the properties of affine Grassmannians and the metric on it. Most of the content in this section is from \cite[Section 1.2]{gan2023hausdorff}. See also in \cite[Section 4.1]{gan2023exceptional}.

\medskip

\subsection{Notation and some useful lemmas}

We will frequently use the following definitions.
\begin{definition}\label{dessetsd1}
For a number $\de>0$ and any set $X$ (in a metric space), we use $|X|_\de$ to denote the maximal number of $\de$-separated points in $X$.
\end{definition}

\begin{definition}\label{dessetsd2}
Let $\de,s>0$. We say $A\subset \R^n$ is a $(\de,s,C)$\textit{-set} if it is $\de$-separated and satisfies the following estimate:
\begin{equation}\label{deltas}
    \# (A\cap B_r(x)) \le C (r/\de)^s.
\end{equation}
for any $x\in \R^n$ and $1\ge r\geq \de$. In this paper, the constant $C$ is not important, so we will just say $A$ is a $(\de,s)$-set if
\[ \#(A\cap B_r(x))\lesssim (r/\de)^s \]
for any $x\in \R^n$ and $1\ge r\geq \de$.
\end{definition}

\begin{remark}
    {\rm
    We remark that we make ``$\de$-separated" as a part of the definition for a $(\de,s)$-set. 
    }
\end{remark}

\begin{lemma}
Let $\de,s>0$ and let $B\subset \R^n$ be any set with $\mathcal H^s_\infty(B) =: \kappa >0$. Then, there exists a $(\de,s)$-set $P\subset B$ with $\#P\gtrsim \kappa \de^{-s}$.
\end{lemma}
\begin{proof}
See \cite[Lemma 3.13]{fassler2014restricted}.
\end{proof}

\begin{lemma}\label{frostmans} Fix $a>0$.
Let $\nu$ be a probability measure satisfying $\nu(B_r)\lesssim r^a$ for any $B_r$ being a ball of radius $r$. If $A$ is a set satisfying $\nu(A)\ge \kappa$ ($\kappa>0$), then for any $\de>0$ there exists a subset $F\subset A$ such that $F$ is a $(\de,a)$-set and $\# F\gtrsim \kappa \de^{-a}$.
\end{lemma}
\begin{proof}
By the previous lemma, we just need to show $\cH^a_\infty(A)\gtrsim \kappa$. We just check it by definition. For any covering $\{B\}$ of $A$, we have
$$\kappa\le \sum_B\nu(B)\lesssim \sum_B r(B)^a. $$
Ranging over all the covering of $A$ and taking infimum, we get $$\kappa\lesssim \cH^a_\infty(A).$$
\end{proof}

\begin{lemma}\label{usefullemma}
Suppose $X\subset [0,1]^2$ with $\dim X< s$. Then for any $\e>0$, there exist dyadic squares $\cC_{2^{-k}}\subset \cD_{2^{-k}}$ $(k>0)$ so that 
\begin{enumerate}
    \item $X\subset \bigcup_{k>0} \bigcup_{D\in\cC_{2^{-k}}}D, $
    \item $\sum_{k>0}\sum_{D\in\cC_{2^{-k}}}r(D)^s\le \e$,
    \item $\cC_{2^{-k}}$ satisfies the $s$-dimensional condition: For $l<k$ and any $D\in \cD_{2^{-l}}$, we have $\#\{D'\in\cC_{2^{-k}}: D'\subset D\}\le 2^{(k-l)s}$.
\end{enumerate}
\end{lemma}

\begin{proof}
    See \cite[Lemma 2]{gan2022restricted}.
\end{proof}

\subsection{Metric on affine Grassmannian}

For every $k$-plane $V\in A(k,n)$, we can uniquely write it as
\[ V=\dir(V)+x_V, \]
where $\dir(V)\in G(k,n)$ and $x_V\in V^\perp$. $\dir(V)$ refers to the direction of $V$, as can be seen that $\dir(V)=\dir(V')\Leftrightarrow V\parallel V'$.

In this paper, we use $\Aloc(k,n)$ to denote the set of $k$-planes $V$ such that $x_V\in B^n(0,1/2)$. ($B^n(0,1/2)$ is the ball of radius $1/2$ centered at the origin in $\R^n$.)
\begin{equation}\label{Aloc}
    \Aloc(k,n)=\{V: V\tx{~is~a~}k\tx{~dimensional~plane~}, x_V\in B^n(0,1/2)\}. 
\end{equation} 
Later in our proof, instead of considering $A(k,n)$, we only care about those $V$ lying near the origin. This is through a standard localization argument.

Next, we discuss the metrics on $G(k,n)$ and $\Aloc(k,n)$.
For $V_1, V_2\in G(k,n)$, we define
\[ d(V_1,V_2)=\|\pi_{V_1}-\pi_{V_2}\|. \]
Here, $\pi_{V_1}:\R^n\rightarrow V_1$ is the orthogonal projection. We have another characterization for this metric. Define $\rho(V_1,V_2)$ to be the smallest number $\rho$ such that $B^n(0,1)\cap V_1\subset N_{\rho}(V_2)$. We have the comparability of $d(\cdot,\cdot)$ and $\rho(\cdot,\cdot)$.

\begin{lemma}
    There exists a constant $C>0$ (depending on $k,n$) such that
    \[ \rho(V_1,V_2)\le  d(V_1,V_2)\le  C\rho(V_1,V_2).  \]
\end{lemma}
\begin{proof}
    Suppose $B^n(0,1)\cap V_1\subset N_\rho(V_2)$, then for any $v\in\R^n$, we have
    \[ |\pi_{V_1}(v)-\pi_{V_2}(v)|\lesssim \rho|v|, \]
    which implies $d(V_1,V_2)\lesssim \rho$. On the other hand, if for any $|v|\le 1$ we have
    \[ |\pi_{V_1}(v)-\pi_{V_2}(v)|\le d|v|, \]
    then we obtain that $\pi_{V_1}(v)\subset N_{d}(V_2)$. Letting $v$ ranging over $B^n(0,1)\cap V_1$, we get $B^n(0,1)\cap V_1\subset N_{d}(V_2)$, which means $\rho(V_1,V_2)\le  d$. 
\end{proof}

We can also define the metric on $\Aloc(k,n)$ given by
\begin{equation}\label{defdist}
    d(V,V')=d(\dir(V),\dir(V'))+|x_{V}-x_{V'}|. 
\end{equation} 
Here, we still use $d$ to denote the metric on $\Aloc(k,n)$ and it will not make any confusion.

Similarly, for $V,V'\in \Aloc(k,n)$ we can define $\rho(V,V')$ to be the smallest number $\rho$ such that $B^n(0,1)\cap V\subset N_\rho(V')$. We also have the following lemma. We left the proof to the interested readers. 

\begin{lemma}\label{comparablelem}
    There exists a constant $C>0$ (depending on $k,n$) such that for $V,V'\in \Aloc(k,n)$,
    \[ C^{-1} d(V,V')\le \rho(V,V')\le C d(V,V').  \]
\end{lemma}

\begin{definition}\label{Vr}
    For $V\in A(k,n)$ and $0<r<1$, we define 
    \[V_r:=N_r(V)\cap B^n(0,1).  \]
We say that $V_r$ is a \textbf{$k$-dimensional~$r$-slab}.
\end{definition}
Actually, $V_r$ is morally a slab of dimensions $\underbrace{r\times \dots\times r}_{k \textup{~times}}\times \underbrace{1\times \dots\times 1}_{n-k \tx{~times}}$. When $k$ is already clear, we simply call $V_r$ an $r$-slab.
If $W$ is a convex set such that $C^{-1} W\subset V_r\subset C W$, then we also call $W$ an $r$-slab. Here, the constant $C$ will be a fixed large constant.

\begin{definition}
    For two $r$-slab $V_r$ and $V'_r$. We say they are comparable if $C^{-1}V_r\subset V'_r\subset C V_r$. We say they are essentially distinct if they are not comparable.
\end{definition}

In this paper, we will also consider the balls and $\de$-neighborhood in $\Aloc(k,n)$. Recall that we use $B_r(x)$ to denote the ball in $\R^n$ of radius $r$ centered at $x$. To distinguish the ambient space, we will use letter $Q$ to denote the balls in $\Aloc(k,n)$. For $V\in \Aloc(k,n)$, we use $Q_r(V)$ to denote the ball in $\Aloc(k,n)$ of radius $r$ centered at $V$. More precisely,
\[ Q_r(V_0):=\{V\in \Aloc(k,n): d(V,V_0)\le r  \}. \]
For a subset $X\subset \Aloc(k,n)$, we use the fancy letter $\cN$ to denote the neighborhood in $\Aloc(k,n)$:
\[ \cN_r(X):=\{V\in \Aloc(k,n): d(V,X)\le r\}. \]
Here, $d(V,X)=\inf_{V'\in X}d(V,V')$.

\bigskip

\subsection{Hausdorff dimension on metric space}

We briefly discuss how to define the Hausdorff dimension for subsets of a metric space. Let $(M,d)$ be a metric space. For $X\subset M$, we denote the $s$-dimensional Hausdorff measure of $X$ under the metric $d$ to be $\cH^s(X;d)$. 
We see that if $d'$ is another metric on $M$ such that $d(\cdot,\cdot)\sim d'(\cdot,\cdot)$, then $\cH^s(X;d)\sim \cH^s(X;d')$. It make sense to define the Hausdorff dimension of $X$ which is independent of the choice of comparable metrics:
\[ \dim X:= \sup\{s: \cH^s(X;d)>0\}. \]

\subsection{\texorpdfstring{$\de$}{}-discretized version: projections of lines onto \texorpdfstring{$k$}{}-planes} We first talk about the projections of points to $k$-planes. Let $x\in B^n(0,1/2)$ and $V$ be a $k$-plane. we see that the orthogonal projection of $x$ onto $V$ is $P_V(x)$, and the fiber of $P_V$ at $P_V(x)$ is the $(n-k)$-plane $P_V^{-1}(P_V(x))$. The $\de$-discretized version is to replace $x$ by a $\de$-ball $B_\de$ centered at $x$, and replace $P_V^{-1}(P_V(x))$ by an $(n-k)$-dimensional $\de$-slab $T=\bigg(P_V^{-1}(P_V(x))\bigg)_\de$. (Here, see Definition \ref{Vr}.)
We see that $T$ is orthogonal to $V$, and $T$ contains $B_\de$.

We want to generalize this $\de$-discretized notion for projections of lines. We meet a new issue about \textit{transversality}: the projection of a line onto a $k$-plane may be a point but not a line. To handle with this degenerate case, we need to restrict ourselves to subsets $A_\mu \subset \Aloc(1,n)$ and $G_\mu\subset G(k,n)$. Here, $0<\mu<1/100$ is a parameter. $A_\mu$ is a $\mu$-ball in $\Aloc(1,n)$ and $G_\mu$ is a $\mu$-ball in $G(k,n)$, so that they satisfy the following quantitative transversality: For any $\ell\in A_\mu$ and $V\in G_\mu$,
\[ \angle(\ell, V^\perp)>\mu. \]
(See also \eqref{transcond} in Theorem \ref{exthm}.) Now for any $\ell\in A_\mu$ and $V\in G_\mu$, we see that $\ell$ is quantitatively away from $V^\perp$, and hence $P_V(\ell)$ is a line in $V$, and there exits an $(n-k+1)$-plane of form $W=P_V(\ell)\oplus W'$ where $W'$ is an $(n-k)$-subspace orthogonal to $V$. We say that $W$ is \textbf{orthogonal to} $V$ \textbf{at} $P_V(\ell)$. More generally, we have the following definition.

\begin{definition}\label{deflineorth}
    Let $V\in G(k,n)$, $\ell'$ be a line in $V$. If an $(n-k+1)$-plane $W$ is of form $W=\ell'\oplus W'$ where $W'$ is an $(n-k)$-subspace orthogonal to $V$, then we say that $W$ is orthogonal to $V$ at $\ell'$.
\end{definition}

\begin{remark}
    {\rm Here are two ways to understand this notion in terms of preimage of $P_V$, or preimage of $\pi_V$.
    On the one hand, we see that 
    \[W=P_V^{-1}(\ell').\]
    On the other hand, we can build a subset 
$\bW$ of $A(1,n)$ from $W$ as
\[ \bW:=\{ \ell\in A(1,n): \ell\subset W, \ell \not\perp V \}. \]
    Then we also have 
    \[ \bW=\pi_V^{-1}(\ell')\cap A(1,n).\]
    }
\end{remark}

Next, we discuss the geometry of the $\de$-discretized version. We will constantly use the following heuristic.\\

\noindent
{\bf Heuristic.} Fix $0<\de<1$. Given $W\in \Aloc(k,n)$ (see \eqref{Aloc}), we have two $\de$-thickened versions of $W$. One is $W_\de$, which is a $k$-dimensional $\de$-slab in $B^n(0,1)$. The other is $Q_\de(W)$, which is a $\de$-ball in $\Aloc(k,n)$. By Lemma \ref{comparablelem} and ignoring some constant, we can morally think of $W_\de$ and $Q_\de(W)$ as follows:
\begin{equation}\label{heuristic}
\begin{split}
    Q_\de(W)&\approx \{W'\in\Aloc(k,n): W'\cap B^n(0,1)\subset W_\de  \},\\  
 W_\de&\approx \bigg(\bigcup_{W'\in Q_\de(W)}W'\bigg)\cap B^n(0,1). 
\end{split}
\end{equation} 
It is good to think of them as the same thing. The reader can consider the two $\de$-thickened versions for a point $x$: both of $Q_\de(\{x\})$ and $\{x\}_\de$ are $B^n(x,\de)$ (the ball of radius $\de$ centered at $x$), since we can identify $\Aloc(0,n)$ as a subset of $\R^n$.
\bigskip

Let us talk about the projections.
As before, let $\ell\in A_\mu$ and $V\in G_\mu$. We have two $\de$-thickened versions of $\ell$: $\ell_\de$ which is a subset of $B^n(0,1)$, and $Q_\de(\ell)$ which is a subset of $\Aloc(1,n)$. By the quantitative transversality condition between $\ell$ and $V$, we see that $P_V(\ell_\de)$ has dimensions $\sim \de\times \dots\times \de\times 1$, where the implicit constant depends on $\mu$. When $\mu$ is fixed, we may just ignore this implicit constant, so let us assume $P_V(\ell_\de)$ is a $\de$-tube. 

We have the $\de$-disretized version for Definition \ref{deflineorth}. Let $W=P_V(\ell)\oplus W'$ which is an $(n-k+1)$-plane orthogonal to $V$ at $P_V(\ell)$. We can morally think of $W_\de$ as $P_V(\ell_\de)\times W'$. We can also morally think of $W_\de$ as $P_V^{-1}(P_V(\ell_\de))\cap B^n(0,1)$. There are some features of $W_\de$: 
\begin{enumerate}
    \item It is an $(n-k+1)$-dimensional $\de$-slab;
    \item It contains $\ell_\de$;
    \item Its intersection with $V$ is a $\de$-tube $P_V(\ell_\de)$, and the other $(n-k)$ directions are orthogonal to $V$.
\end{enumerate}
We would like to say $W_\de$ is \textbf{orthogonal to} $V$ \textbf{at} $P_V(\ell_\de)$. See Figure \ref{orthatline}.

\begin{figure}[ht]
\centering
\begin{minipage}[b]{0.85\linewidth}
\includegraphics[width=11cm]{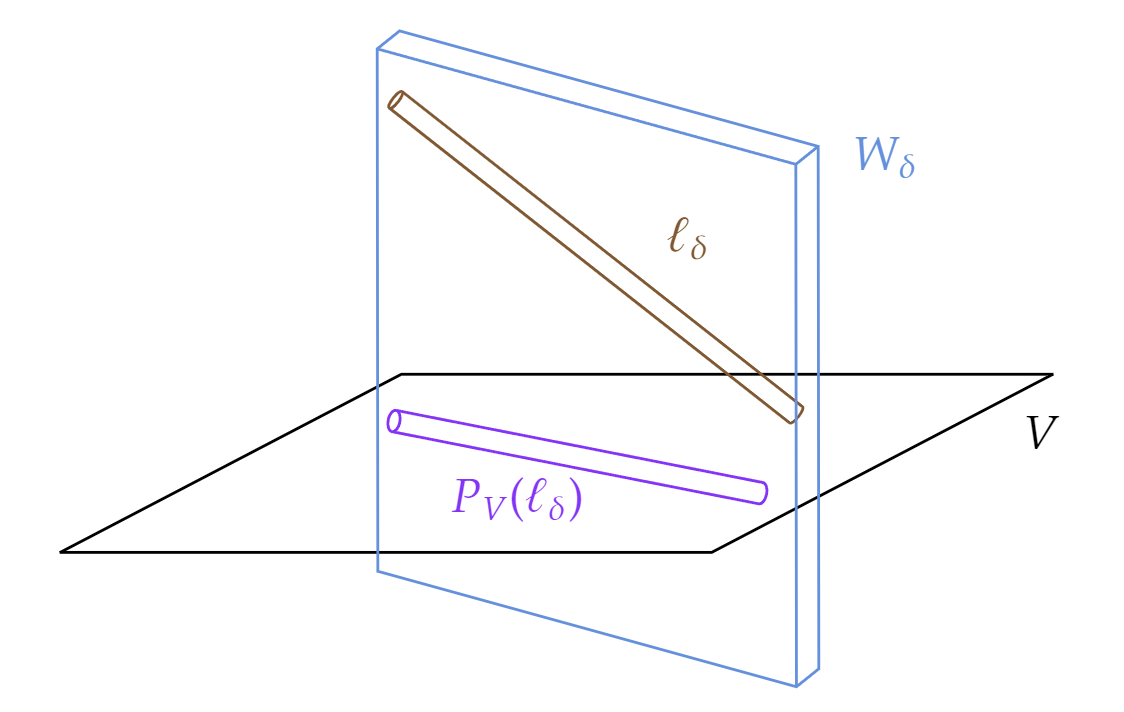}
\caption{}
\label{orthatline}
\end{minipage}
\end{figure}

\section{Theorem \ref{exthm} implies Proposition \ref{lowerprop}}\label{sec3}

We will use the Kaufman-type estimate \eqref{kaufman} to deduce \eqref{lowerineq1}, and use Falconer-type estimate \eqref{falconer} to deduce \eqref{lowerineq2}.

\medskip

We first prove \eqref{lowerineq1}. Since $S(a)$ is monotonically increasing in $a$, we can assume $a\in [0,k-1]$ and to prove 
\[ S(a)\ge a. \]
Suppose by contradiction that $S(a)\le a-2\e$ for some $\e>0$. From the definition of $S(a)$ in \eqref{defSa}, we can find a set $\bA\subset A(1,n)$ with $\dim(\bA)=a$ so that the set
\begin{equation}\label{stilltrue}
    E=\{V\in G(k,n): \dim(\pi_V(\bA))< a-\e\}
\end{equation}
has positive measure. Since $A(1,n)$ can be covered by countably many translated copies of $\Aloc(1,n)$, so we may assume $\bA\subset \Aloc(1,n)$ with $\dim(\bA)>a-\e/100$ and \eqref{stilltrue} still has positive measure.

Since $\dim(\bA)>a-\e/100$, there exists $\ell\in\bA$ such that 
\[ \dim(Q_r(\ell)\cap \bA)\ge a-\e/2, \]
for any $r>0$ and $Q_r(\ell)$ being a ball of radius $r$ centered at $\ell$ in the metric space $\Aloc(1,n)$. Usually, $\ell$ is referred to as an $(a-\e/2)$-density point in $\bA$. Without loss of generality, we assume $\ell$ is parallel to the $x_n$-axis. We also use $\ell_0$ to denote the translation of $\ell$ to the origin, which by our assumption, is the $x_n$-axis.

Since $\{V\in G(k,n): \ell_0 \subset V^\perp \}$ has zero measure as a subset of $G(k,n)$, we see that
\[ E\cap \{ V\in G(k,n): \ell_0\not \subset V^\perp  \}  \]
has positive measure. For any $V\in \{V\in G(k,n): \ell_0 \subset V^\perp \}$, there exists a number $\mu=\mu_V>0$ such that the following holds. Let $Q_{\mu}(\ell)\subset \Aloc(1,n)$ be a ball of radius $\mu$ centered at $\ell$, and $Q_{\mu}(V)\subset G(k,n)$ be a ball of radius $\mu$ centered at $V$. Then 
\[ \angle(\ell_1,V_1^\perp)>\mu \]
for any $\ell_1\in Q_{\mu}(\ell)$ and $V_1\in Q_{\mu}(V)$. This guarantees the transversality condition \eqref{transcond}. 
We will let $s=a-\e$. We can check
\[ s<\min\{ \dim(Q_r(\ell)\cap\bA)),2(k-1) \}. \] 
Then, we can apply Theorem \ref{exthm} to the set $\bA\cap Q_{\mu}(\ell)$. 
By \eqref{kaufman}, we obtain that
\begin{equation}
\begin{split}
    \dim(\{ V\in Q_\mu(V): \dim(\pi_V(\bA\cap Q_{\mu}(\ell)))<a-\e \})&\le k(n-k)+a-\e-(k-1)\\
    &\le k(n-k)-\e.
\end{split}
\end{equation}
The last inequality is because $a\le k-1$.
This implies
\begin{equation}\label{localupper}
    \dim(\{ V\in Q_\mu(V): \dim(\pi_V(\bA))<a-\e \})\le k(n-k)-\e. 
\end{equation} 
Since we can write $E\cap \{ V\in G(k,n): \ell_0\not\subset V^\perp \}$ as a countable union of the set of form
\[ \{ V\in Q_{\mu}(V): \dim(\pi_V(\bA))<a-\e \}. \]
Therefore, $\dim(E\cap \{ V\in G(k,n): \ell_0\not\subset V^\perp \})\le k(n-k)-\e$, which contradicts that $E\cap \{ V\in G(k,n): \ell_0\not\subset V^\perp \}$ has positive measure.

\medskip

Next. we prove \eqref{lowerineq2}. We can assume $a\in [n-1,n+k-2]$ and to prove
\[ S(a)\ge a-(n-k). \]
Arguing in the same way, we can find a set $\bA\subset \Aloc(1,n)$ with $\dim(\bA)>a-\e/100$ so that
\[ E=\{V\in G(k,n): \dim(\pi_V(\bA))<a-(n-k)-\e\} \]
has positive measure. By finding an $(a-\e/2)$-density point $\ell$ in $\bA$ and similarly defining $\ell_0$, we can assume
\[ E\cap \{V\in G(k,n):\ell_0\not\subset V^\perp\} \]
has positive measure.
We will let $s=a-(n-k)-\e$. We can check
\[ s<\min\{ a-\e/2,2(k-1) \}. \]
Then, we can apply Theorem \ref{exthm} and obtain a similar estimate as \eqref{localupper}:
\begin{equation}
\begin{split}
   &\ \ \ \dim(\{ V\in Q_\mu(V): \dim(\pi_V(\bA))\le a+(n-k)-\e \})\\
    &\le k(n-k)+(a-(n-k)-\e)-(a-\e/2)+(n-k)\\
    &\le k(n-k)-\e/2.
\end{split}
\end{equation}
Therefore, $\dim(E\cap \{ V\in G(k,n): \ell_0\not\subset V^\perp \})\le k(n-k)-\e/2$, which contradicts that $E\cap \{ V\in G(k,n): \ell_0\not\subset V^\perp \}$ has positive measure.

\begin{remark}
    {\rm
    We actually showed that the number
\begin{equation}
    \wt S(a):=\inf_{\bA\subset A(1,n), \dim(\bA)=a}\esssup_{V\in G(k,n)}\dim\Big(\pi_V(\bA)\cap A(1,V)\Big), 
\end{equation}
which looks smaller than $S(a)$ defined in \eqref{defSa}, actually has the same lower bound as $S(a)$ in Proposition \ref{lowerprop}.    
    }
\end{remark}

\section{Kaufman-type exceptional set estimate}\label{sec4}

We state a discretized version of the estimate \eqref{kaufman}.

\begin{theorem}\label{diskaufman}
Fix a number $0<\mu<1/100$, and let $A_\mu, G_\mu$ be as in Theorem \ref{exthm}.
    Fix $t>k(n-k)-(k-1)$, $0<s<a$ and $0<u<t-k(n-k)+(k-1)$. For sufficiently small $\e>0$ (depending on $\mu,n,k,a,s,t,u$), the following holds. Let $0<\de<1/100$. Let $\bH\subset A_\mu$ be a $(\de,u)$-set with $\#\bH\gtrsim (\log\de^{-1})^{-2}\de^{-u}$. Let $\cV\subset G_\mu$ be a $(\de,t)$-set with $\#\cV\gtrsim (\log\de^{-1})^{-2}\de^{-t}$. Suppose for each $V\in G(k,n)$, we have a collection of slabs $\T_V$, where each $T\in \T_V$ has dimensions $\underbrace{\de\times\dots\times \de}_{k-1\ \textup{times}}\times \underbrace{1\times \dots\times 1}_{n-k+1 \ \textup{times}}$ and is orthogonal to $V$ at some $\de$-tube in $V$. 
    
    The $s$-dimensional Forstman condition holds for $\T_V$: for any $\de\le r\le 1$ and any $(n-k+1)$-dimensional $r$-slab $W_r$ that is orthogonal to $V$ at some $r$-tube in $V$, we have
    \begin{equation}\label{sfrostman}
        \#\{T\in \T_V: T\subset W_r  \}\lesssim (r/\de)^s.
    \end{equation}
    
    We also assume that for each $\ell \in\bH$,
    \begin{equation}\label{tlower}
        \#\{ V\in \cV: \ell_\de\subset T, \textup{~for~some~}T\in\T_V  \}\gtrsim (\log\de^{-1})^{-2}\#\cV.
    \end{equation}
    Then, we have
    \begin{equation}\label{kaufest}
        \de^{-u-t}\lesssim_\e \de^{-\e/2}\sum_{V\in\cV}\#\T_V\lesssim \de^{-\e-t-s}.
    \end{equation}
    
\end{theorem}

\subsection{Proof of the Kaufman-type estimate}
We give the proof of \eqref{kaufman} using Theorem \ref{diskaufman}. We choose $\al<\dim(\bA), t<\dim(E_s(\bA))$. By Frostman's lemma, there exist probability measures $\nu_{\bA}$ supported on $\bA$ and $\nu_E$ supported on $E_s(\bA)$ satisfying the Frostman's condition:
\begin{align}
    &\nu_\bA(Q_r)\lesssim r^\al \textup{~for~any~} Q_r \textup{~being~a~ball~of~radius~} r \textup{~in~}\Aloc(1,n). \\
    &\nu_E(Q_r)\lesssim r^t \textup{~for~any~} Q_r \textup{~being~a~ball~of~radius~} r \textup{~in~}G(k,n).
\end{align}

We only need to prove 
\begin{equation}
    t\le k(n-k)+s-(k-1),
\end{equation}
since then we can let $t\rightarrow \dim(E_s(\bA))$.

Fix a $V\in E_s(\bA)$. By definition, we have $\dim (\pi_V(\bA))<s$. We also fix a small number $\en$ which we will later send to $0$.
We view $\pi_V(\bA)$ as a subset of $\Aloc(1,V)=\Aloc(1,k)$. By Lemma \ref{usefullemma}, we can cover $\pi_V(\bA)$ by balls $\D_V=\{D\}$ in $\Aloc(1,V)$, each of which has radius $2^{-j}$ for some integer $j>|\log_2\en|$. We define $\mathbb D_{V,j}:=\{D\in\mathbb D_V: r(D)=2^{-j}\}$.
Lemma \ref{usefullemma} yields the following properties:
\begin{equation}\label{rsless1}
    \sum_{D\in\mathbb D_V}r(D)^s<1,
\end{equation}
and for each $j$ and $r$-ball $Q_r\subset \Aloc(1,V)$, we have
\begin{equation}\label{structure}
    \#\{D\in \D_{V,j}: D\subset Q_r\}\lesssim \left(\frac{r}{2^{-j}}\right)^s.
\end{equation}
On the one hand, $D$ is a $2^{-j}$-ball in $\Aloc(1,V)$. On the other hand we can view $D$ as a $2^{-j}$-tube in $V$. We use $t_D$ to denote this $2^{-j}$-tube. By the heuristic \eqref{heuristic}, we can view $t_D$ as
\[ \bigcup_{\ell\in D}\ell \cap \{x\in V: |x|\le 1\}. \]

For each $V\in E_s(A)$, we can find such a $\D_V=\bigcup_j\D_{V,j}$. We also define the slabs $\T_{V,j}:=\{\pi^{-1}_V(t_D)\cap B^n(0,1): D\in\D_{V,j}\}$, $\T_{V}=\bigcup_j\T_{V,j}$. Each plank in $\T_{V,j}$ has dimensions 
\[
\underbrace{2^{-j}\times 2^{-j} \times \dots \times 2^{-j}}_{k-1 \text{~times}}\times \underbrace{1 \times 1 \times \dots \times 1}_{n-k+1 \text{~times}}
\]
such that it is orthogonal to $V$ at some $t_D$ for $D\in\D_{V,j}$. For each such slab $T\in \T_{V,j}$, we use its bold-font $\bT$ to denote the set of lines whose unit truncations are in $T$. More precisely,
\begin{equation}
    \bT:=\{\ell: \ell\cap B^n(0,1) \subset T \}.
\end{equation}

One easily sees that $\bA\subset \bigcup_{T\in \T_V}\bS $. By pigeonholing, there exists $j(V)$ such that
\begin{equation}\label{kaufpigeon1}
    \nu_\bA(\bA\cap(\cup_{T\in\T_{V,j(V)}}\bT ))\ge \frac{1}{10j(V)^2}\nu_\bA(\bA)=\frac{1}{10j(V)^2}.
\end{equation} 
For each $j>|\log_2\en|$, define $E_{s,j}(A):=\{V\in E_s(A): j(V)=j\}$. Then we obtain a partition of $E_s(\bA)$:
\[
E_s(\bA)=\bigsqcup_j E_{s,j}(\bA).
\]
By pigeonholing again, there exists $j$ such that
\begin{equation}\label{kaufpigeon2}
    \nu_E(E_{s,j}(\bA))\ge \frac{1}{10j^2}\nu_E(E_s(\bA))= \frac{1}{10j^2}. 
\end{equation} 
In the rest of the proof, we fix this $j$. We also set $\de=2^{-j}$. By Lemma \ref{frostmans}, there exists a $(\de,t)$-set $\mathcal V\subset E_{s,j}(\bA)$ with cardinality $\#\mathcal V\gtrsim (\log\de^{-1})^{-2}\de^{-t}$.

Next, we consider the set $M:=\{(\ell,V)\in \bA\times \mathcal V: \ell\in\cup_{T\in\T_{V,j}}\bT \}$. We also use $\#$ to denote the counting measure on $\mathcal V$.
Define the sections of $M$:
$$ M_\ell=\{V: (\ell,V)\in M\},\ \ \  M_V:=\{\ell: (\ell,V)\in M\}. $$
By \eqref{kaufpigeon1} and Fubini, we have
\begin{equation}\label{kaufpigeon3}
    (\nu_\bA\times \#)(M)\ge \frac{1}{10j^2} \mu(\mathcal V).
\end{equation}
This implies
\begin{equation}\label{kaufpigeon4}
    (\nu_\bA\times \#)\bigg(\Big\{(\ell,V)\in M: \mu(M_\ell)\ge\frac{1}{20j^2}\mu(\mathcal V)  \Big\}\bigg)\ge \frac{1}{20j^2}\mu(\mathcal V).
\end{equation} 
By \eqref{kaufpigeon4}, we have
\begin{equation}\label{kaufpigeon5}
    \nu_\bA\bigg(\Big\{\ell\in \bA: \mu(M_\ell)\ge \frac{1}{20j^2}\mu(\mathcal V) \Big\}\bigg)\ge \frac{1}{20j^2}. 
\end{equation} 

We are ready to apply Theorem \ref{diskaufman}. Recall $\de=2^{-j}$ and $\#\mathcal V\gtrsim (\log\de^{-1})^{-2}\de^{-t}$.
We may assume $t>k(n-k)-(k-1)$, otherwise we are done. Set 
\begin{equation}\label{defu}
    u=\min\{t-k(n-k)+(k-1),a\}-\e.
\end{equation}
By \eqref{kaufpigeon5} and Lemma \ref{frostmans}, we can find a $(\de,u)$-subset of $\{\ell\in \bA: \# M_\ell\ge \frac{1}{20j^2}\#\mathcal V \}$ with cardinality $\gtrsim (\log\de^{-1})^{-2}\de^{-u}$. We denote this set by $\bH$. For each $\ell\in H$, we see that there are $\gtrsim (\log\de^{-1})^{-2}\#\mathcal V$ many slabs from $\cup_{V\in\mathcal V}\T_{V,j}$ that contain $\ell_\de$. We can now apply Theorem \ref{diskaufman} to obtain
\[
\de^{-u-t}\lesssim_\e \de^{-\e-t-s} .
\]
By letting $\en\rightarrow 0$ (and hence $\de\rightarrow 0$), we obtain
\[ u+t\le t+s+\e. \]
Plugging in the definition of $u$ and letting $\e\rightarrow 0$, we obtain
\[ \min\{t-k(n-k)+(k-1),a\}\le s. \]
Since $a>s$, we obtain
\begin{equation}
    t\le k(n-k)+s-(k-1).
\end{equation}

\subsection{Discretized Kaufman-type estimate}
\begin{proof}[Proof of Theorem \ref{diskaufman}]

Define $\T=\bigcup_{V\in\cV} \T_V$, where $\T_V$ is given by Theorem \ref{diskaufman}. If two $\de$-slab are comparable, then we just identify them. So we assume the $\de$-slabs in $\T$ are essentially distinct.

We define the following incidence pair:
\begin{equation}
    \cI=\cI(\T,\cV):=\{(T,V)\in \T\times \cV: T\in \T_V  \}.
\end{equation}

We will prove the theorem by comparing the upper and lower bound of $\cI$. We easily see the upper bound
\begin{equation}\label{kauflowerbound}
    \#\cI \lesssim \sum_{V\in\cV}\#\T_V\lesssim  \#\cV \#\T_V\lesssim \de^{-t-s}.
\end{equation} 
Here, we used the bound $\#\cV\lesssim \de^{-t}$ since $\cV$ is a $(\de,t)$-set, and we used $\#\T_V\lesssim \de^{-s}$ by plugging in $r=1$ in \eqref{sfrostman}.

\medskip

For the lower bound of $\#\cI$, we will use the following inequality:
\[ \#(\bigcup_i A_i)\ge \sum_i \#(A_i\setminus \bigcup_{j\neq i} A_j). \]

By \cite[Lemma 13]{bright2022exceptional},
we choose a $\de|\log\de|^{O(1)}$-separated subset $\bH'\subset \bH$ such that $\bH'$ is a $(\de|\log\de|^{O(1)},u)$-set and \begin{equation}
    \#\bH'\gtrsim |\log\de|^{-O(1)}\de^{-u}.
\end{equation}
Here, $O(1)$ is a large constant to be determined later.

For each $\ell\in \bH'$, we define the following subset of $\cI$.
\begin{equation}
    \cI_\ell:=\{(T,V)\in \cI: \ell_\de\subset T\}.
\end{equation}
Then, $\cI\supset \bigcup_{\ell\in\bH'} \cI_\ell$.

We have
\begin{align}
    \label{furstenberg}\#\cI\ge \#\left(\bigcup_{\ell\in\bH'}\cI_\ell\right) &\geq \#\left(\bigcup_{\ell\in \bH'} \left(\cI_\ell \setminus \bigcup_{\ell_1\in \bH'\setminus \{\ell\}} \cI_{\ell_1}\right)\right) \\
    &= \sum_{\ell\in \bH'} \#\left(\cI_\ell \setminus \bigcup_{\ell_1\in \bH'\setminus \{\ell\}} \cI_{\ell_1}\right) \\
    &= \sum_{\ell\in \bH'} \bigg(\# \cI_\ell - \sum_{\ell_1\in \bH'\setminus \{\ell\}}\#\left(\cI_\ell \cap  \cI_{\ell_1}\right)\bigg).
\end{align}
We will show that
$$ \# \cI_\ell - \sum_{\ell_1\in \bH'\setminus \{\ell\}}\#\left(\cI_\ell \cap  \cI_{\ell_1}\right)\ge \frac{1}{2}\#\cI_\ell. $$
This will imply
\begin{equation}
    \#\cI\ge \frac{1}{2}\sum_{\ell\in\bH'}\#\cI_\ell\gtrsim (\log\de^{-1})^{-O(1)}  \de^{-u-t}.
\end{equation}

We make one observation for $\cI_\ell$. Note that for any $V\in\cV$, there are $\lesssim 1$ $T_i$'s from $\T$ such that $(T_i,V)\in \cI_\ell$. The reason is that if $T_i\in\cI_\ell$, then $T_i$ is orthogonal to $V$ at $P_V(\ell_\de)$. There can be at most $\lesssim 1$ such $T_i$'s. By losing a constant factor in the estimate, we may assume for any $V\in\cV$ there is one or no $T$ such that $(T,V)\in\cI_\ell$. Therefore, we can identify $\cI_\ell$ with the set on the left hand side of \eqref{tlower}. Also, \eqref{tlower} implies for any $\ell\in \bH'$,
\begin{equation}
    \#\cI_{\ell}\gtrsim (\log\de^{-1})^{-2}\#\cV\gtrsim (\log\de^{-1})^{-4} \de^{-t}.
\end{equation}

\medskip

For fixed $\ell$, and $\ell_1\in \bH'\setminus \{\ell\}$, we want to find an upper bound for $\#(\cI_\ell\cap\cI_{\ell_1})$. 

By the condition that $\bH'$ is contained in a $\mu$-ball in $\Aloc(1,n)$, we may assume all the $\ell$ in $\bH'$ form an angle $\le 10^{-1}$ with the $x_n$-axis. We will introduce a new metric for these lines. It will not be hard to see that this metric is comparable to the metric on $\Aloc(1,n)$ introduced in \eqref{defdist}.

Let $\Pi_0=\R^{n-1}\times \{0\}, \Pi_1=\R^{n-1}\times \{1\}$. Then each $\ell\in\bH'$ will intersects $\Pi_0$, $\Pi_1$ at two points denoted by $P_0(\ell), P_1(\ell)$. We will use $(P_0(\ell), P_1(\ell))$ for the local coordinates of $\ell$. And we define the metric to be 
\begin{equation}
    d(\ell,\ell_1):=|P_0(\ell)-P_0(\ell_1)|+|P_1(\ell)-P_1(\ell_1)|.
\end{equation}
This metric is equivalent to the metric defined in \eqref{defdist}. We leave out the proof.

Next, we will estimate $\#(\cI_\ell\cap\cI_{\ell_1})$. By the definition of $d(\ell,\ell_1)$, we may assume $|P_0(\ell)-P_0(\ell_1)|\ge \frac{1}{2}d(\ell,\ell_1)$. We denote the number
\[ d:= |P_0(\ell)-P_0(\ell_1)|.\]
We remind the reader $\de/2\le d\le 1$.

We introduce the two dimensional plank $\cP$ which has size $\sim 1\times d$. The side of $\cP$ of length $1$ is $\ell\cap B^n(0,1)$. The side of $\cP$ of length $d$ is $\overrightarrow{P_0(\ell)P_0(\ell_1)}$. 

If $(T,V)\in \cI_\ell\cap \cI_{\ell_1}$, then $\cP\subset T$ and $P_V(T)$ is contained in a $\de$-tube. Therefore, $P_V(\cP)$ is contained in a $\de$-tube.
We formally write
\[ \cI_\ell\cap \cI_{\ell_1}\supset \{ V\in\cV: P_V(\cP) \textup{~is~contained~in~a~} \de\textup{-tube} \}. \]

We make the following observation.
\begin{lemma}\label{numberofplane}
Let $\cM$ be a set of $k$-dimensional subspaces $V$ so that the projection of $\cP$ to $V$ is a line or point. In other words,
\begin{equation}\label{defM}
    \cM:=\{ V\in G(k,n): \dim(P_V(\cP))\le 1 \}. 
\end{equation} 
Then we have
\begin{equation}\label{eqcountinglem}
    \{ V\in G(k,n): P_V(\cP) \textup{~is contained in a~} \de\textup{-tube} \}
    \subset \cN_{C\de d^{-1}}(\cM).
\end{equation}
Here, $C$ is some big constant.
\end{lemma}
\begin{proof}
    We sketch the idea of the proof. Let $V$ be in the LHS of \eqref{eqcountinglem}. Then $P_V(\cP)$ is contained in a $1\times \de$-tube. Define $\cQ$ to be a $1\times 1$-square, which is obtained by prolonging the length-$d$ side of $\cP$ $d^{-1}$ times. We see that $P_V(\cQ)$ is contained in a $1\times d^{-1}\de$-tube. The next step is a rotation argument. We claim that by rotating $V$ within angle $\lesssim d^{-1}\de$, we obtain another $k$-subspaces $W$ such that $P_W(\cQ)$ is a line. We leave out the proof of this claim.

    Therefore, we find a $W\in \cM$ such that $d(V,W)\lesssim d^{-1}\de$, which finishes the proof of lemma. 
\end{proof}

The next lemma is about the dimension of $\cM$.
\begin{lemma}\label{dimlem}
    Let $\cM$ be given by \eqref{defM}. Then
    \begin{equation}
        \dim(\cM)=1+k(n-k-1).
    \end{equation}
\end{lemma}

\begin{proof}
    We sketch the proof. We may assume $\cP$ span the plane $\R^2$.  If $\dim(P_V(\R^2))\le 1$, then $V^\perp\cap \R^2$ contains a line or the whole $\R^2$. We just need to consider the first case. We think of $V^\perp$ as spanned by $n-k$ orthonormal vectors $v_1,\dots,v_{n-1}$. From the condition that $V^\perp \cap \R^2$ contains a line, we have one dimensional choice of vector $V_1$ that lies in $\R^2$, and then we choose $v_2,\dots,v_{n-k}$ from $v_1^\perp$ which has $\dim(G(n-k-1,n-1))=k(n-k-1)$ dimensional choice. Therefore, we see that such $V^\perp$ and hence $V$ has $1+k(n-k-1)$ dimensional choice.
\end{proof}

We have the estimate
\begin{equation}
    \#(\cI_\ell\cap\cI_{\ell_1})\le \#(\cV\cap \cN_{C\de d^{-1}}(\cM)).
\end{equation}
And we remind the reader that $\cV$ is a $(\de,t)$-set.

By Lemma \ref{dimlem},
we can cover $\cN_{C\de d^{-1}}(\cM)$ by $\sim (\de^{-1}d)^{1+k(n-k-1)}$ many $\de d^{-1}$-balls in $G(k,n)$.
By the $(\de,t)$ property of $\cV$, we have
$$ \#(\cV \cap \cN_{C\de d^{-1}}(\cM) )\lesssim (\de^{-1}d)^{1+k(n-k-1)} d^{-t}.  $$
So, we have
\begin{align*}
    \sum_{\ell_1\in \bH'\setminus\{\ell\}}\#(\cI_\ell\cap\cI_{\ell_1})&\lesssim \sum_{\ell_1\in \bH'\setminus\{\ell\}} \left(\de^{-1}d(\ell,\ell_1)\right)^{1+k(n-k-1)} d(\ell,\ell_1)^{-t}\\
    &=\sum_{\de|\log\de|^{O(1)}\le d\le 1}\ \ \sum_{\ell_1\in \bH', d(\ell,\ell_1)\sim d}\left(\de^{-1} d\right)^{1+k(n-k-1)} d^{-t}.
\end{align*} 
Here the summation over $d$ is over dyadic numbers.
Since $\#(\bH'\cap Q_d(\ell))\lesssim (\frac{d}{\de})^{u}$, the expression above is bounded by
\begin{align*}
    &\lesssim\sum_{\de|\log\de|^{O(1)}\le d\le 1}\left(\frac{d}{\de}\right)^{u}\left(\frac{d}{\de}\right)^{1+k(n-k-1)} d^{-t}\\
    &=\de^{-t}\sum_{\de|\log\de|^{O(1)}\le d\le 1}\left(\frac{d}{\de}\right)^{u+1+k(n-k-1)-t}\\
    &\lesssim \de^{-t}|\log\de|^{O(1)(u+1+k(n-k-1)-t)}.
\end{align*}
From the condition in Theorem \ref{diskaufman}, $u+1+k(n-k-1)-t\le -\e<0$. by choosing the constant $O(1)$ big enough, we have
$$ \sum_{\ell_1\in \bH'\setminus\{\ell\}}\#(\cI_\ell\cap\cI_{\ell_1})\le C^{-1}|\log\de|^{-4}\de^{-t}\le \frac{1}{2}\#\cI_\ell. $$
As a result, we have
$$ \#\cI\ge |\log\de|^{-O(1)}\de^{-u-t}. $$
Compared with the upper bound of $\#\cI$ in \eqref{kauflowerbound}, we finish the proof.
\end{proof}

\section{Falconer-type exceptional set estimate}\label{sec5}

We will use the high-low method to prove Falconer-type exceptional set estimate. The key ingredient in the high-low method is the Fourier analysis. Since we are working with the set $\bA$ as a subset of $A(1,n)$, we need the Fourier transform in $A(1,n)$. It is quite hard to define a global Fourier transform on $A(1,n)$, but since our set $\bA$ has been localized in $A_\mu$ (see Theorem \ref{exthm}), we just need to consider the Fourier transform in the set $\wt A(1,n)$ consisting of lines in $A(1,n)$ that is transverse to $\R^{n-1}$. In other words,
\begin{equation}
    \wt A(1,n):=\{\ell\in A(1,n): \ell \textup{~is~not~parallel~to~} \R^{n-1} \}.
\end{equation}

In the next subsection, we introduce the Fourier transform in $\wt A(1,n)$ and discuss some properties. 

\subsection{Fourier transform on \texorpdfstring{$A(1,n)$}{}}\hfill

First, we introduce the coordinate for $\wt A(1,n)$. By definition, every $\ell\in\wt A(1,n)$ intersect $\R^{n-1}\times \{0\}$ and $\R^{n-1}\times \{1\}$ at two points for which we denote by $X(\ell),Y(\ell)$. We see that we parametrize $\wt A(1,n)$ using the coordinates $(X,Y)$.
\begin{equation}\label{definverse}
    (X,Y): \wt A(1,n)\xrightarrow{\cong} \R^{n-1}\times \R^{n-1}\ \ \ell\mapsto (X(\ell),Y(\ell)).
\end{equation}
Here, we can just view $\R^{n-1}\times \R^{n-1}$ as $\R^{2(n-1)}$.
We see that we can pull back the Fourier transform in $\R^{2(n-1)}$ to that in $\wt A(1,n)$. For convenience, we denote the inverse of \eqref{definverse} to be
\begin{equation}\label{lfunction}
    \ell: \R^{n-1}\times \R^{n-1}\xrightarrow{\cong} \wt A(1,n)\ \ (X,Y)\mapsto \ell(X,Y).
\end{equation}

Here is the precise definition for Fourier transform on $\wt A(1,n)$. Suppose $F=F(\ell)$ is a function on $\wt A(1,n)$, we define the Fourier transform of $F$, denoted by $\wh F(\ell)$, to be the function on $\wt A(1,n)$ given by
\begin{equation}
    \wh F(\ell):=\big(\cF(F\circ\ell)\big)(X(\ell),Y(\ell)).
\end{equation}
We explain the expression. Given $F$ as a function on $\wt A(1,n)$, we use $\ell$ in \eqref{lfunction} to pull back $F$ to a function $F\circ \ell$ which is defined on $\R^{2(n-1)}$. We apply $\cF$, which is the standard Fourier transform on $\R^{2(n-1)}$, to $F\circ \ell$, and obtain the function $\cF(F\circ \ell)$ on $\R^{2(n-1)}$. At last, we use $(X,Y)$ in \eqref{definverse} to pull back $\cF(F\circ \ell)$ to the function defined on $\wt A(1,n)$. 

\bigskip

Next, we will introduce the dual rectangle. This tool is always used as a black box. It says that if a function is a smooth bump function adapted to a rectangle, then its Fourier transform is morally a smooth function adapted to the dual rectangle multiplied by some normalizing constant. We first review this property for Fourier transform in $\R^n$ and then talk about it in $\wt A(1,n)$.

\begin{definition}
    Let $R\subset \R^n$ be a rectangle of dimensions $a_1\times a_1\times \dots\times a_n$. We define the dual rectangle of $R$ to be another rectangle $R^*$ centered at the origin with dimensions $a_1^{-1}\times a_2^{-1}\times \dots\times a_n^{-1}$, so that the  edge of $R^*$ of length $a_i^{-1}$ is parallel to the  edge of $R$ of length $a_i$.
\end{definition}

We state with out proof the following result.
\begin{proposition}
    Let $R$ be a rectangle in $\R^n$. Then there exists a smooth bump function $\psi_R$ such that $\psi_R(x)\ge 1$ for $x\in R$ and $\psi_R$ decays rapidly outside $R$. And $\psi_R$ satisfies $\supp(\cF(\psi_R))\subset R^*$ and
    \begin{equation}\label{dualineq}
        |\cF (\psi_R)| \lesssim |R|\cdot 1_{R^*},
    \end{equation}
Intuitively, we may think of $\psi_R$ as the indicator function $1_R$.
\end{proposition}

\medskip

\begin{figure}[ht]
\centering
\begin{minipage}[b]{0.85\linewidth}
\includegraphics[width=11cm]{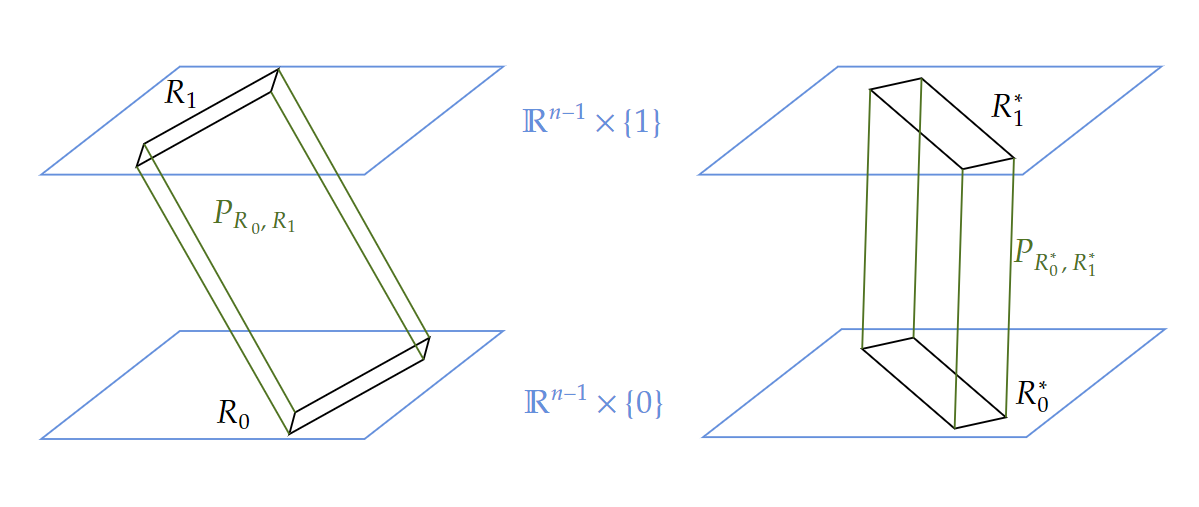}
\caption{}
\label{grassrectangle}
\end{minipage}
\end{figure}

Now, we are going to introduce the notion of rectangle and dual rectangle in $\wt A(1,n)$, and obtain the inequality of form \eqref{dualineq}. We will view $\wt A(1,n)$ as $\R^{n-1}\times \R^{n-1}$.

\begin{definition}\label{defrect}
    We say $\bR$ is a rectangle in $ \R^{n-1}\times \R^{n-1}$ if $\bR$ has the following form
    \begin{equation}
        \bR=R_0\times R_1.
    \end{equation}
 Here, $R_0\subset \R^{n-1}$ is a rectangle and $R_1\subset \R^{n-1}$ is a translated copy of $R_0$. We define the dual rectangle to be \begin{equation}
     \bR^*=R_0^*\times R_1^*.
 \end{equation}
\end{definition}

\begin{remark}
    {\rm See Figure \ref{grassrectangle}.
    Since $R_1$ is a translated copy of $R_0$, we have $R_0^*=R_1^*$. There are two ways to understand $\bR$. On the one hand, $\bR$ is a Cartesian product of $R_0, R_1$ as a subset of $\R^{n-1}\times \R^{n-1}$. On the other hand, $\bR$ is a subset of $\wt A(1,n)$. If we use $P_{R_0,R_1}$ to denote the rectangle in $\R^n$ which is the convex hull of $R_0$ and $R_1$, then $\bR= \bigg\{\ell \in \wt A(1,n): \ell\cap \{0\le x_n\le 1\} \subset P_{R_0,R_1}\bigg\} $.
    }
\end{remark}

We have the following result.
\begin{proposition}
    Let $\bR$ be a rectangle in $\R^{n-1}\times \R^{n-1}$. Then there exists a smooth bump function $\psi_\bR$ such that $\psi_\bR(X,Y)\ge 1$ for $(X,Y)\in\bR$ and $\psi_\bR$ decays rapidly outside $\bR$. And $\psi_\bR$ satisfies $\supp(\wh \psi_\bR)\subset \bR^*$ and
    \begin{equation}
        |\wh \psi_\bR| \lesssim |\bR|\cdot 1_{\bR^*},
    \end{equation}
Intuitively, we may think of $\psi_\bR$ as the indicator function $1_\bR$.
\end{proposition}

\subsection{Discretized estimate}
We state a discretized version of \eqref{falconer}. The setup is basically the same as Theorem \ref{diskaufman}. Instead of \eqref{kaufest}, we have \eqref{falest}.

\begin{theorem}\label{disfalconer}
Fix a number $0<\mu<1/100$, and let $A_\mu, G_\mu$ be as in Theorem \ref{exthm}.
    Fix $t>0$, $0<s<a$. For sufficiently small $\e>0$ (depending on $\mu,n,k,a,s,t$), the following holds. Let $0<\de<1/100$. Let $\bH\subset A_\mu$ be a $(\de,a)$-set with $\#\bH\gtrsim (\log\de^{-1})^{-2}\de^{-a}$. Let $\cV\subset G_\mu$ be a $(\de,t)$-set with $\#\cV\gtrsim (\log\de^{-1})^{-2}\de^{-t}$. Suppose for each $V\in \cV$, we have a collection of slabs $\T_V$, where each $T\in \T_V$ has dimensions $\underbrace{\de\times\dots\times \de}_{(k-1)\ \textup{times}}\times \underbrace{1\times \dots\times 1}_{(n-k+1)\ \textup{times}}$ and is orthogonal to $V$ at some $\de$-tube in $V$.

    The $s$-dimensional Forstman condition holds for $\T_V$: for any $\de\le r\le 1$ and any $(n-k+1)$-dimensional $r$-slab $W_r$ that is orthogonal to $V$ at some $r$-tube in $V$, we have
    \begin{equation}\label{sfrostman2}
        \#\{T\in \T_V: T\subset W_r  \}\lesssim (r/\de)^s.
    \end{equation}
    
    We also assume that for each $\ell \in\bH$,
    \begin{equation}\label{low2}
        \#\{ V\in \cV: \ell_\de\subset T, \textup{~for~some~}T\in\T_V  \}\gtrsim (\log\de^{-1})^{-2}\#\cV.
    \end{equation}
    Then, we have
    \begin{equation}\label{falest}
        \de^{-t-a}\lesssim_\e \de^{-\e-(k+1)(n-k)-s}.
    \end{equation}
    
\end{theorem}

The proof of \eqref{falconer} is deduced from Theorem \ref{disfalconer} by a similar argument. We just do not repeat here. We will focus on the proof of Theorem \ref{disfalconer}.

\begin{proof}[Proof of Theorem \ref{disfalconer}]
    The idea is to use the Fourier transform on affine Grassmannian together with the high-low method.

We set up some notation. Let $\bH_\de$ be the set of $\de$-balls in $\Aloc(1,n)$ whose centers are points in $\bH$. We use $\ell$ to denote the elements in $\bH$. By the heuristic \eqref{heuristic}, we just denote the $\de$-ball in $\bH_\de$ by the corresponding bold-font 
\[\bell_\de:=\{\ell' \in\Aloc(1,n): d(\ell,\ell')\le \de\}.\] 
By \eqref{sfrostman2}, for each $V\in \cV$, we have $s$-dimensional set of slabs $\T_V=\{T\}$. For a slab $T$, we use the bold-font $\bT$ to denote a subset of $\wt A(1,n)$, so that each $\ell\in\bT$ satisfies $\ell\cap B^n(0,1)\subset T$. Again, here we use the heuristic \eqref{heuristic}.   
    
Consider the the following integral
\begin{equation}
    \int_{\bH_\de}(\sum_{V\in\cV}\sum_{T\in \T_V}1_\bT)^2.
\end{equation}

First of all, by \eqref{low2} we notice that each $\ell_\de \in\bH_\de $ is contained in $\gtrsim (\log \de^{-1})^{-2} \de^{-t}$ different $\bT$. The volume of a $\de$-ball in $A(1,n)$ is $\sim \de^{2(n-1)}$, therefore the volume of $\bH_\de$ is $\sim \de^{2(n-1)}\#\bH$
We have the lower bound
\begin{equation}
    \int_{\bH_\de}(\sum_{V\in\cV}\sum_{T\in \T_V}1_\bT)^2\gtrsim (\log\de^{-1})^{-2t}\de^{2(n-1)}\#\bH\de^{-2t}\gtrsim_\e \de^{\e}\de^{2(n-1)}\de^{-a}\de^{-2t}.
\end{equation}

\medskip

Our next goal is to obtain an upper bound for the integral. Recalling Definition \ref{defrect}, since each $\bT$ is a rectangle in $\Aloc(1,n)$, we write $\bT=R_{0,\bT}\times R_{1,\bT}$. Here, $R_{0,\bT}$ is a rectangle in $\R^{n-1}\times \{0\}$ of dimensions $\underbrace{\de\times\dots \times \de}_{k-1\  \text{times}}\times \underbrace{1\times\dots \times 1}_{n-k\  \text{times}}$, $R_{1,\bT}$ is a translated copy of $R_{0,\bT}$ which lies in $\R^{n-1}\times \{1\}$. We can choose a smooth bump function 
\begin{equation}\label{bumpfunc}
    \psi_\bT(X,Y)=\psi_{R_{0,\bT}}(X)\psi_{R_{1,\bT}}(Y)
\end{equation}
adapted to $\bT$ so that $\supp \wh \psi_\bT\subset \bT^*=R_{0,\bT}^*\times R_{0,\bT}^*$. 

Define
\[ f_V=\sum_{T\in \T_V}\psi_\bT \text{~and~} f=\sum_{V\in \cV}f_V. \]
We will do the high-low decomposition. Let $K=(\log\de^{-1})^{O(1)}$, where $O(1)$ is a large number to be determined later. 

Since we can identify $\wt A(1,n)$ and $\R^{n-1}\times \R^{n-1}$ through the coordinates in \eqref{definverse} and \eqref{lfunction}, we will constantly jump between these two spaces in the latter discussion.

Choose a function $\eta(X,Y)$ on $\wt A(1,n)=\R^{n-1}\times \R^{-1}$, such that $\eta(X,Y)$ is a smooth bump function adapted to $B^{2(n-1)}(0,(K\de)^{-1})$. We have the following high-low decomposition for $f_V$:
\[ f_V=f_{V,\text{high}}+f_{V,\text{low}}, \]
where $\wh f_{V,\text{low}}=\eta \wh f_V$ and $\wh f_{V,\text{high}}=(1-\eta)\wh f_V$.
For each $(X,Y)\in \bH_\de$, we have
\begin{equation}
    (\log \de^{-1})^{-2}\#\cV \lesssim f(X,Y)\le |f_{\text{high}}(X,Y)|+|f_{\text{low}}(X,Y)|.
\end{equation}
We also let
\begin{equation}
    f_{\textup{high}}=\sum_{V\in\cV} f_{V,\textup{high}},\ \ \ f_{\textup{low}}=\sum_{V\in\cV} f_{V,\textup{low}}.
\end{equation}

We will show that the high part dominates for $(X,Y)\in \bH_\de$, i.e., $|f_{\text{high}}(X,Y)|\gtrsim (\log\de^{-1})^{-2}\#\cV$. It suffices to show
\begin{equation}\label{lowbound}
    |f_{\text{low}}(X,Y)|\le C^{-1} (\log\de^{-1})^{-2}\#\cV,
\end{equation}
for a large constant $C$.

Recall that $f_{\text{low}}=\sum_{V\in\cV}f_V*\eta^\vee$. By the definition of $\eta$, we see that $\eta^\vee$ is an $L^1$-normalized bump function essentially supported in $B^{2(n-1)}(0,K\de)$ and decays rapidly out side of it. Let $\chi(X)$ be a positive function $=1$ on $B^{n-1}(0,K\de)$ and decays rapidly outside $B^{n-1}(0,K\de)$. We have
\begin{equation}
    |\eta^\vee(X,Y)|\lesssim (K\de)^{-2(n-1)}\chi(X)\chi(Y).
\end{equation}
Together with \eqref{bumpfunc}, we have 
\begin{equation}\label{plugback}
    \begin{split}
        |f_{\text{low}}(X,Y)|&\lesssim
        \sum_{V\in\cV}\sum_{T\in\T_V} \psi_\bT* \eta^\vee (X,Y) \\
        &\lesssim (K\de)^{-2(n-1)}\sum_{V\in\cV}\sum_{T\in\T_V}\psi_{R_{0,\bT}}*\chi(X) \psi_{R_{1,\bT}}*\chi(Y)
    \end{split}
\end{equation}
Let $K R_{0,\bT}$ be a tube of dimensions $\underbrace{K\de\times\dots \times K\de}_{k-1\  \text{times}}\times \underbrace{1\times\dots \times 1}_{n-k\  \text{times}}$, which is the $K$-dilation of the short edges of $R_{0,\bT}$. Define $K R_{1,\bT}$ similarly.
Let $K\bT=KR_{0,\bT}\times KR_{1,\bT}$.
We see that $\psi_{R_{0,\bT}}*\chi(X)$ is morally a bump function at $K R_{0,\bT}$ with weight $(K\de)^{n-1}K^{-(k-1)}$. Let us just ignore the rapidly decaying tail and write
\begin{equation}
    \psi_{R_{0,\bT}}*\chi(X) \lesssim (K\de)^{n-1}K^{-(k-1)} 1_{K R_{0,\bT}}(X).
\end{equation}
Similarly, we have
\begin{equation}
    \psi_{R_{1,\bT}}*\chi(X) \lesssim (K\de)^{n-1}K^{-(k-1)} 1_{K R_{1,\bT}}(X).
\end{equation}

Plugging back to \eqref{plugback}, we obtain
\begin{equation}
    |f_{\text{low}}(X,Y)|\lesssim K^{-2(k-1)} \sum_{V\in\cV}\sum_{T\in\T_V} 1_{K\bT}(X,Y).
\end{equation}

Fix $\ell=\ell(X,Y)\in \bH$, let $W_{K\de}$ be the $K\de$-slab that is orthogonal to $V$ at $P_V(\ell_{K\de})$. 
By the $s$-dimensional condition for $\T_V$, we have
\begin{equation}
    \sum_{T\in \T_V}1_{K\bT}(X,Y)\lesssim \#\{ T\in\T_V: T\subset W_{K\de} \}\lesssim K^s.
\end{equation}
Plugging back to \eqref{plugback}, we get
\begin{equation}
    |f_{\text{low}}(X,Y)|\lesssim K^{s-2(k-1)}\#\cV.
\end{equation}
Noting that $s<2(k-1)$, we may choose $K\sim (\log \de^{-1})^{O(1)}$ for large $O(1)$ so that \eqref{lowbound} holds.

\medskip

Next, we want to regroup the slabs in $\T_V$. Note that each $T\in \T_V$ is an $(n-k+1)$-dimensional $\de$-slab, by the transversality assumption, its intersection with $\R^{n-1}\times \{0\}$ and $\R^{n-1}\times \{1\}$ are two congruent $(n-k)$-dimensional $\de$-slabs. We define $\G$ to be a maximal $\de$-separated subset of $G(n-k,n-1)$. We are going to define $\T_{V}^W$ where $W$ ranges over $\G$. For each $T\in \T_{V}$, we put $T$ into $\T_{V}^W$ if $T\cap (\R^{n-1}\times \{0\})$ is parallel to $W$ up to $\de$-error.
We obtain a partition
\begin{equation}\label{observation}
    \T_V=\bigsqcup_{W\in \G} \T_{V}^W. 
\end{equation} 
We have the following observation. 
For $T_1\in \T_{V_1}^W, T_2\in \T_{V_2}^W$, we have that if $T_1$ and $T_2$ are not comparable, then $\bT_1\cap \bT_2=\emptyset$. The reason is that if $T_1$ and $T_2$ are not comparable, then either $T_1\cap (\R^{n-1}\times \{0\})$ and $T_2\cap (\R^{n-1}\times \{0\})$ are disjoint or $T_1\cap (\R^{n-1}\times \{1\})$ and $T_2\cap (\R^{n-1}\times \{1\})$ are disjoint, which means $\bT_1\cap \bT_2=\emptyset$.

For $T\in\T_V^W$, we have
\begin{equation}
    \supp(\wh{\psi_\bT})\subset W_\de^*\times W_\de^*.
\end{equation}
Here, $W_\de^*$ is the dual rectangle of $W_\de$ in $\R^{n-1}$. Therefore, we see that $W_\de^*$ has dimensions $\underbrace{\de^{-1}\times\dots \times \de^{-1}}_{k-1\  \text{times}}\times \underbrace{1\times\dots \times 1}_{n-k\  \text{times}}$.

Now, we have
\begin{equation}\label{comb1}
    \int_{\bH_\de}(\sum_{V\in\cV}\sum_{T\in \T_V}1_\bT)^2\lesssim\int_{\bH_\de}(\sum_{V\in\cV}\sum_{T\in \T_V}\psi_\bT)^2\lesssim \int_{\bH_\de} |f_{\text{high}}|^2\lesssim \int_{\wt A(1,n)} |f_{\text{high}}|^2.
\end{equation}

By Plancherel, it is further bounded by
\begin{equation}\label{ineq1}
  \int_{\wt A(1,n)} |\sum_{V\in\cV}\sum_{W\in\G} \wh{f_{V}^W} (1-\eta)|^2= \int_{\wt A(1,n)} | \sum_{W\in\G}\sum_{V\in \cV}\wh{f_{V}^W} (1-\eta)|^2.
\end{equation}

We will estimate the overlap of $\{\supp(\sum_{V\in\cV}\wh f_{V}^W (1-\eta))\}_{W\in\G}$. We notice that \[\supp(\sum_{V\in\cV}\wh {f_{V}^W}(1-\eta))\subset \bigg(W_\de^*\times W_\de^*\bigg)\setminus B^{2(n-1)}(0,(K\de)^{-1}).  \]

We pick $(\Xi_1,\Xi_2)\in \wt A(1,n)$ in the frequency space. We assume $(\Xi_1,\Xi_2)$ lies in some $\supp(\sum_{V\in\cV}\wh {f_{V}^W}(1-\eta))$. Therefore, we can assume without loss of generality that $\Xi_1\notin B^{n-1}(0,\frac{1}{2}(K\de)^{-1})$. For any $W\in \G$, if $(\Xi_1,\Xi_2)\in \supp(\sum_{V\in\cV}\wh {f_{V}^W}(1-\eta))$, then $\Xi_1\in W^*_\de$. Therefore, the overlap of $\{\supp(\sum_{V\in\cV}\wh f_{V}^W (1-\eta))\}_{W\in\G}$ at $(\Xi_1,\Xi_2)$ is bounded by
\begin{equation}\label{1overlapbound}
    \sum_{W\in \G}1_{W_\de^*}(\Xi_1),
\end{equation}
Since $\Xi_1\notin B^{n-1}(0,\frac{1}{2}(K\de)^{-1})$, we can further bound \eqref{1overlapbound} by the overlapping number of
\begin{equation}\label{overlapnumber2}
    \bigg\{W_\de^*\setminus B^{n-1}(0,\frac{1}{2}(K\de)^{-1})\bigg\}_{W\in\G}.
\end{equation}

If we do a dilation by the factor $\de$, then $W_\de^*$ becomes $(W^\perp)_\de$. So we just need to bound the overlapping number of
\begin{equation}\label{overlapnumber}
    \bigg\{(W^\perp)_\de\setminus B^{n-1}(0,\frac{1}{2}K^{-1})\bigg\}_{W\in\G}.
\end{equation}

We observe that when $W$ ranges over $\G$, $W^\perp$ will range over $\de$-separated subset of $G(k-1,n-1)$. By \cite[Lemma 18]{bright2022exceptional}, we see that the overlapping number of \eqref{overlapnumber} is bounded by 
\begin{equation}
    K^{O(1)}\de^{-\dim(G(k-2,n-2))}
\end{equation}

The RHS of \eqref{ineq1} is bounded by
\begin{equation}\label{comb2}
\begin{split}
    &K^{O(1)}\de^{-\dim(G(k-2,n-2))}\int_{\wt A(1,n)} \sum_{W\in\G}| \sum_{V\in \cV}\wh{f_{V}^W} (1-\eta)|^2\\
    \le &K^{O(1)}\de^{-\dim(G(k-2,n-2))}\int_{\wt A(1,n)} \sum_{W\in\G}| \sum_{V\in \cV}\wh{f_{V}^W} |^2\\
    =&K^{O(1)}\de^{-\dim(G(k-2,n-2))}\sum_{W\in\G}\int_{\wt A(1,n)} | \sum_{V\in \cV}f_{V}^W|^2
\end{split}
\end{equation}
The last step is by Plancherel to return back to the physical space.

Now, for fixed $W\in\G$, we estimate
\begin{equation}\label{ineq2}
    \int_{\wt A(1,n)} | \sum_{V\in  \cV}f_{V}^W|^2=\int_{\wt A(1,n)} | \sum_{V\in \cV}\sum_{T\in\T_V^W}\psi_\bT|^2.
\end{equation}
We may ignore the rapidly decaying tail of $\psi_\bT$ and think of it as $1_\bT$. To estimate the RHS of \eqref{ineq2},
we will estimate the overlapping number of $\{\sum_{T\in\T_V^W}1_\bT\}_{V\in\cV}$. 

By an observation we addressed in the paragraph below \eqref{observation}, for any two tubes $T_1\in\T_{V_1}^W, T_2\in \T_{V_2}^W$, we have either $T_1$ and $T_2$ are comparable, or $\bT_1$ and $\bT_2$ are disjoint. Our task is as follows. For a given $T$, how many $V\in\cV$ can there be so that $T$ is comparable to some $T'\in \T_V^W$? Noting that if $T$ is comparable to some element in $\T_V^W$, then $T$ must be morally orthogonal to $V$ at some $\de$-tube. Finding such $V$ is equivalent to finding $\de$-separated lines in $G(1,T)$. Therefore, for fixed $T$, the number of such $V\in\cV$ is $\lesssim \de^{-(\dim T-1)}= \de^{-(n-k)}$. As a result, we obtain
\begin{equation}\label{comb3}
    \int_{\wt A(1,n)} | \sum_{V\in \cV}\sum_{T\in\T_V^W}1_\bT|^2\lesssim \de^{-(n-k)}\int_{\wt A(1,n)} \sum_{V\in \cV}\sum_{T\in\T_V^W}1_\bT.
\end{equation}

Combining \eqref{comb1},\eqref{ineq1},\eqref{comb2},\eqref{ineq2} and \eqref{comb3}, we get the upper bound
\begin{equation}\label{ineq3}
    \begin{split}
        \int_{\bH_\de}(\sum_{V\in\cV}\sum_{T\in\T_V}1_T)^2\lesssim K^{O(1)}\de^{-(k-2)(n-k)}\de^{-(n-k)}\int_{\wt A(l,n)} \sum_{V\in\cV}\sum_{T\in\T_V}1_\bT.
    \end{split}
\end{equation}
We notice that for any $(n-k+1)$-dimensional slab $S$, the dimension of
\[ \{\ell\in \wt A(1,n): \ell\cap B^n(0,1)\subset S  \} \]
is $2(n-k)$. If $S$ is the core of the $\de$-slab $T$, or in other words $T=S_\de$, then
\[ \bT=\{\ell\in \wt A(1,n): \ell\cap B^n(0,1)\subset T  \} \]
is roughly the $\de$-neighborhood of $\{\ell\in \wt A(1,n): \ell\cap B^n(0,1)\subset S  \}$ in $\wt A(1,n)$, which
has measure $\sim \de^{2(n-1)-2(n-k)}$. We also note that $K\sim (\log\de^{-1})^{O(1)}$. We can bound the RHS of \eqref{ineq3} by
\begin{equation}
    \lesssim_\e  \de^{-\e}\de^{-(k-2)(n-k)}\de^{-(n-k)} \de^{-s-t}\de^{2(n-1)-2(n-k)}
\end{equation}

In summary, we obtain
\begin{equation}
     \int_{\bH_\de}(\sum_{V\in\cV}\sum_{T\in\T_V}1_\bT)^2\lesssim_\e \de^{-\e}\de^{-(k-2)(n-k)}\de^{-(n-k)} \de^{-s-t}\de^{2(n-1)-2(n-k)}.
\end{equation}

Comparing with the lower bound
\begin{equation}
    \int_{\bH_\de}(\sum_{V\in\cV}\sum_{T\in \T_V}1_\bT)^2\gtrsim \de^{2(n-1)}\#\bH\de^{-2t}\sim \de^{2(n-1)}\de^{-a}\de^{-2t},
\end{equation}
we obtain
\begin{equation}
    t\le k(n-k)+s-a+(n-k).
\end{equation}
\end{proof}

\section{The projection problem for \texorpdfstring{$A(l,n)$}{}}

It is quite natural to further think about the same projection problem for $A(l,n)$. We formally state the problem.

Fix integers $0<l<k<n$.
Let $V\in G(k,n)$.
Note that for $L\in A(l,n)$, we have $P_V(L)\in A(j,V)$ for some $0\le j\le k$. We can define
\[ \pi_{l,V}: A(l,n)\rightarrow \bigsqcup_{j=0}^l A(j,V) \]
\[ L\mapsto P_V(L). \]

We can ask the same Marstrand-type projection problem. For $0<a<(l+1)(n-l)=\dim(A(l,n))$, what is the optimal number $s(n,k,l,a)$ such that the following is true? Let $\bA\subset A(l,n)$ with $\dim(\bA)=a$. Then we have
\begin{equation}
    \dim(\pi_{l,V}(\bA))=s(n,k,l,a),\ \textup{for~a.e.~}V\in G(k,n).
\end{equation}

For simplicity, we assume $n,k,l$ are fixed. So we write $\pi_V$ for $\pi_{l,V}$.

\begin{definition}
    Fix $0\le l< k<n$. For any $0<a<\dim(A(l,n))$, define
    \begin{equation}
        S(a):=\inf_{\bA\subset A(l,n), \dim(\bA)=a}\esssup_{V\in G(k,n)}\dim(\pi_V(\bA)).
    \end{equation}
Here, we require $\bA$ to be a Borel set to avoid some measurability issue. We also remark that $S(a)=S_{l,k,n}(a)$ should also depend on $l,k,n$, but we just omit them from the notation for simplicity.
\end{definition}

A reasonable conjecture would be:

\begin{conjecture}
    For $j=0,1,\dots,l$, we have the following lower bounds of $S(a)$. 
\begin{equation}
    \begin{split}
        S(a)= a-(l-j)(n-k)\ \ \ a\in [(l-j)(n-l),(l-j)(n-l)+k-l].\\
        S(a)= (l-j+1)(k-l)\ \ \ a\in [(l-j)(n-l)+k-l,(l-j+1)(n-l)].
    \end{split}
\end{equation}
\end{conjecture}

We are able to show the upper bound of $S(a)$ by constructing examples as follows.

\begin{proposition}
    For $j=0,1,\dots,l$, we have the following lower bounds of $S(a)$. 
\begin{equation}\label{propineq}
    \begin{split}
        S(a)\le a-(l-j)(n-k)\ \ \ a\in [(l-j)(n-l),(l-j)(n-l)+k-l].\\
        S(a)\le (l-j+1)(k-l)\ \ \ a\in [(l-j)(n-l)+k-l,(l-j+1)(n-l)].
    \end{split}
\end{equation}
\end{proposition}

\begin{proof}
Fix $j\in \{0,\dots,l\}$. Let $V_j=\R^j\times \{0\}^{n-j}$, $V_l=\R^l\times \{0\}^{n-l}$, $V_k=\R^k\times \{0\}^{n-k}$. We have $V_j\subset V_l\subset V_k\subset \R^n$. We introduce a new notation. For $j<l$ and $V\in G(j,n)$, define
\begin{equation}
\Bush(l,V):=\{W\in G(l,n): W\supset V\},    
\end{equation}
which is the set of $l$-subspaces that contain $V$. We call $V$ the \textit{stem} of the bush. It is not hard to see for $V\in G(j,n)$,
\begin{equation}
    \dim(\Bush(l,V))=(l-j)(n-l).
\end{equation}

We first consider the case $a=(l-j)(n-l)+b$ where $0\le b\le k-l$. We choose $B\subset V_l^\perp$, so that $\dim(B)=b$. Note that this is allowable since $\dim(V_l^\perp)=n-l\ge b$. Next, we choose
\begin{equation}\label{exA}
    \bA=\bigcup_{v\in B} \Bush(l, V_j+v).
\end{equation}
One way to think about $\bA$ is as follows. We first locate the set of stems $\{V_j+v\}_{v\in B}$, and then enrich each stem to become a bush $\Bush(l,V_j+v)$. $\bA$ is the union of these bushes.

We compute the dimension of $\bA$. The idea is that we have $b$-dimensional set of bushes and each bush has dimension $(l-j)(n-l)$, so we have 
\[\dim \bA\le \dim B+\dim(\Bush(l,V_j))= (l-j)(n-l)+b.\] 
The equality holds if there is not too much overlap between different bushes. We handle it in the following way. Consider
\begin{equation}
    \Bush'(l,V_j+v):=\{W\in\Bush(l,V_j+v): W\cap V_l^\perp=\{0\}  \},
\end{equation}
which are the $l$-planes in the bush that are transverse to $V_l^\perp$. 
It is not hard to see 
\[\dim(\Bush'(l,V_j+v))=\dim(\Bush(l,V_j+v))=(l-j)(n-l).\]
We also have that for $v_1\neq v_2\in B$,
\begin{equation}
    \Bush'(l,V_j+v_1)\cap \Bush'(l,V_j+v_2)=\emptyset.
\end{equation}
We prove it. Suppose there exists $W\in \Bush'(l,V_j+v_1)\cap \Bush'(l,V_j+v_2)$. Then $W\supset V_j+v_1, V_j+v_2$ which implies $W\supset v_1-v_2$. However $v_1-v_2\in V_l^\perp$ which contradicts $W\cap V_l^\perp=\{0\}$.

Now we have 
\begin{equation}
    \bA\supset \bigsqcup_{v\in B}\Bush'(l, V_j+v), 
\end{equation}
so $\dim(\bA)\ge \dim B+\dim(\Bush'(l,V_j+v))=  (l-j)(n-l)+b$.
Therefore, we have
\begin{equation}
    \dim(\bA)=(l-j)(n-l)+b=a.
\end{equation}

Next, we show that for generic $V\in G(k,n)$,
\begin{equation}\label{exineq1}
    \dim(\pi_V(\bA))\le (l-j)(k-l)+b=a-(l-j)(n-k).
\end{equation}
We need the following notation. For subspaces $W_1\subset W_2$ with $\dim W_1\le i\le \dim W_2$, define
\begin{equation}
    \Bush(l,W_1,W_2):=\{W\in G(l,W_2): W\supset W_1 \}.
\end{equation}
One can check that $\Bush(l,W_1,\R^n)=\Bush(l,W_1)$ and $\Bush(l,\{0\},W_2)=A(l,W_2)$.

We prove \eqref{exineq1}. If $V\in G(k,n)$ satisfies $V_j\cap V^\perp=\{0\}$ (or in other words, $\pi_V(V_j)$ is $j$-dimensional), then
for each bush $\Bush(l,V_j+v)$ in $\bA$, we have 
\begin{equation}\label{exproject}    \pi_V(\Bush(l,V_j+v))=\Bush(l,\pi_V(V_j+v),V)\sqcup\bigsqcup_{j\le i\le l-1}\Bush(i,\pi_V(V_j+v),V).
\end{equation}
In other words, $\pi_V$ project the bush $\Bush(l,V_j+v)$ to the set of planes of dimension less than $l$ in $V$ that contain $\pi_{V}(V_j+v)$.
Since the first term on the right hand side of \eqref{exproject} dominates the Hausdorff dimension, we have
\begin{equation}
\dim(\pi_V(\Bush(l,V_j+v)))=\dim(\Bush(l,\R^j,\R^k))=(l-j)(k-l).
\end{equation}
Therefore,
\begin{equation}
    \dim(\pi_V(\bA))\le \dim B+(l-j)(k-l)=b+(l-j)(k-l),
\end{equation}
which finishes the proof of \eqref{exineq1}, and hence the first part of \eqref{propineq}.

\medskip

We prove the second part of \eqref{propineq}. We just need to show that for $a=(l-j+1)(n-l)$. The idea is similar to \eqref{exA}, but here we choose $B=V_l^\perp$. We construct
\begin{equation}
    \bA=\bigcup_{v\in V_l^\perp} \Bush(l, V_j+v).
\end{equation}
By the same reasoning, we have
\[ \dim(\bA)=\dim (V_l^\perp)+ (l-j)(n-l)=(l-j+1)(n-l). \]

On the other hand, for generic $V\in G(k,n)$, $\pi_V(\bA)$ consists of $(k-l)$-dimensional set of bushes, since the stems of these projected bushes are parallel $l$-planes in $V$. Also, the dimension of each bush is $(l-j)(k-l)$. Therefore, we have
\begin{equation}
    \dim(\pi_V(\bA))\le (k-l)+(l-j)(k-l)=(l-j+1)(k-l).
\end{equation}
This finishes the second part of \eqref{propineq}.

\end{proof}

\bigskip

However, with the current method in the paper, we are not able to show the lower bound of $S(a)$ for all range of $a$. We can also prove Falconer-type and Kaufman-type exceptional set estimates, and obtain some partial results. We just state it without proof.

\begin{theorem}
    For $j=0$ or $l$, we have the following lower bound of $S(a)$.
    \begin{align}
    S(a)\ge a-(l-j)(n-k)\ \ \ a\in [(l-j)(n-l),(l-j)(n-l)+k-l].\\
    S(a)\ge (l-j+1)(k-l)\ \ \ a\in [(l-j)(n-l)+k-l,(l-j+1)(n-l)].
\end{align}
\end{theorem}
$j=0$ corresponds to the Falconer-type estimate, while $j=l$ corresponds to the Kaufman-type estimate.

\bibliographystyle{abbrv}
\bibliography{refs}

\begin{thebibliography}{10}

\bibitem{bright2022exceptional}
P.~Bright and S.~Gan.
\newblock Exceptional set estimates for radial projections in $\mathbb{R}^n$.
\newblock {\em arXiv preprint arXiv:2208.03597}, 2022.

\bibitem{falconer1982hausdorff}
K.~Falconer.
\newblock Hausdorff dimension and the exceptional set of projections.
\newblock {\em Mathematika}, 29(1):109--115, 1982.

\bibitem{fassler2014restricted}
K.~F{\"a}ssler and T.~Orponen.
\newblock On restricted families of projections in $\mathbb{R}^3$.
\newblock {\em Proceedings of the London Mathematical Society},
  109(2):353--381, 2014.

\bibitem{gan2023exceptional}
S.~Gan.
\newblock Exceptional set estimate through {B}rascamp-{L}ieb inequality.
\newblock {\em arXiv preprint arXiv:2308.07675}, 2023.

\bibitem{gan2023hausdorff}
S.~Gan.
\newblock Hausdorff dimension of unions of $k$-planes.
\newblock {\em arXiv preprint arXiv:2305.14544}, 2023.

\bibitem{gan2022restricted}
S.~Gan, S.~Guo, L.~Guth, T.~L. Harris, D.~Maldague, and H.~Wang.
\newblock On restricted projections to planes in $\mathbb{R}^3$.
\newblock {\em arXiv preprint arXiv:2207.13844}, 2022.

\bibitem{kaufman1968hausdorff}
R.~Kaufman.
\newblock On hausdorff dimension of projections.
\newblock {\em Mathematika}, 15(2):153--155, 1968.

\bibitem{marstrand1954some}
J.~M. Marstrand.
\newblock Some fundamental geometrical properties of plane sets of fractional
  dimensions.
\newblock {\em Proceedings of the London Mathematical Society}, 3(1):257--302,
  1954.

\bibitem{mattila2015fourier}
P.~Mattila.
\newblock {\em {F}ourier analysis and {H}ausdorff dimension}, volume 150.
\newblock Cambridge University Press, 2015.

\bibitem{peres2000smoothness}
Y.~Peres and W.~Schlag.
\newblock Smoothness of projections, {B}ernoulli convolutions, and the
  dimension of exceptions.
\newblock {\em Duke Mathematical Journal}, 2000.

\end{thebibliography}

\end{document}